\documentclass{iopart}
\usepackage{iopams}

\usepackage{hyperref}
\usepackage{graphicx}
\usepackage{mathptmx}
\usepackage[ngerman,american]{babel}
\usepackage{cite}
\usepackage{url}
\usepackage{textcomp}
\usepackage{wasysym}
\usepackage{numprint}

\usepackage[dvipsnames,svgnames]{pstricks}
\usepackage{pstricks-add}
\usepackage{pst-3d}
\usepackage{pst-3dplot}
\usepackage{pst-all, pst-eucl,pst-plot,pst-math,pst-poly,pst-node}

\usepackage[T1]{fontenc}
\usepackage[utf8]{inputenc}

\usepackage{amssymb}
\usepackage{amsthm}

\usepackage{array}
\usepackage{hyperref}
\usepackage[ngerman,american]{babel}
\usepackage{url}

\usepackage[noline]{algorithm2e}
\SetArgSty{mbox}
\SetCommentSty{textrm}

\usepackage[dvipsnames,svgnames]{pstricks}
\usepackage{pstricks-add}
\usepackage{pst-3d}
\usepackage{pst-3dplot}
\usepackage{pst-all, pst-eucl,pst-plot,pst-math,pst-poly,pst-node}

\newcommand{\graphnodes}{\psset{fillstyle=solid,fillcolor=white,radius=4pt,shadow=true,shadowcolor=black!50}}

\newcommand{\cube}[1]{{\pstThreeDBox[fillstyle=solid,fillcolor=orange,linestyle=solid,RotZ=0](#1)(0,0,1)(0,1,0)(1,0,0)}}
\newcommand{\ignore}[1]{}
\newcommand{\DX}{\mathrm{DX}}

\newtheorem{thm}{Theorem}

\newcommand{\eqref}[1]{(\ref{#1})}
\newcommand{\text}[1]{\mbox{#1}}

\begin{document}

\renewcommand{\theenumi}{(\arabic{enumi})}

\title{The Perimeter of Proper Polycubes}

\author{Sebastian Luther$^1$ and Stephan Mertens$^{1,2}$}

\address{\selectlanguage{ngerman}{$^1$Institut\ f"ur\ Theoretische\ Physik,
    Otto-von-Guericke Universit"at, PF~4120, 39016 Magdeburg,
    Germany}} 

\address{$^2$Santa Fe Institute,
1399 Hyde Park Rd,
Santa Fe, NM 87501,
USA}

\ead{SebastianLuther@gmx.de, mertens@ovgu.de}

\begin{abstract}
 We derive formulas for the number of polycubes of size $n$ and
 perimeter $t$ that are proper in $n-1$ and $n-2$ dimensions.
 These formulas complement computer based enumerations of 
 perimeter polynomials in percolation problems. We demonstrate this by
 computing the perimeter polynomial for $n=12$ in arbitrary dimension $d$.
\end{abstract}




\section{Introduction}
\label{sec:intro}

A $d$-dimensional polycube of size $n$ is a set of face-connected
cells in the lattice $\mathbb{Z}^d$. Figure~\ref{fig:polycubes} shows 
all $3$-dimensional polycubes of size $4$. 
Polycubes are a classical topic in recreational
mathematics and combinatorics \cite{golomb:book}. In statistical
physics and percolation theory, polycubes are called
\emph{lattice animals} \cite{stauffer:aharony:book,guttmann:LNP}.

In percolation theory one is interested in counting lattice animals of
a given size $n$ according to their perimeter $t$, i.e., to the number
of adjacent cells that are empty (see Figure~\ref{fig:polycubes}).
If each cell of the lattice is occupied independently with probability $p$, the
average number of clusters of size $n$ per lattice site reads
\begin{equation}
  \label{eq:def-ns}
  \sum_{t} g_{n,t}^{(d)} p^n(1-p)^t\,,
\end{equation}
where $g_{n,t}^{(d)}$ denotes the number of fixed $d$-dimensional polycubes
of size $n$ and perimeter $t$. Fixed polycubes are considered distinct if
they have different shapes or orientations. The $g$'s define the 
\emph{perimeter polynomials} 
\begin{equation}
  \label{eq:perimeter-polynomial}
  P_n^{(d)}(q) = \sum_t g_{n,t}^{(d)} q^t\,. 
\end{equation}
The perimeter polynomials comprise a considerable amount of
information about the percolation problem in $\mathbb{Z}^d$.


Some polycubes that contribute to $g_{n,t}^{(d)}$ span less than $d$
dimensions. A polycube that spans $i$ dimensions is called
\emph{proper} in $i$ dimensions. See Figure~\ref{fig:polycubes} for
examples. Let $G_{n,t}^{(i)}$ denote the number of fixed polycubes of
size $n$ and that are proper in dimension $i$ and that have perimeter
$t$ in $\mathbb{Z}^i$. We then have
\begin{equation}
  \label{eq:lunnon-g}
  g_{n,t}^{(d)} = \sum_{i=1}^{n-1} {d \choose i}\, G_{n,t-2(d-i)n}^{(i)}\,,
\end{equation} 
which expresses the fact that we can choose the $i$ dimensions of a
proper polycube from the $d$ dimensions of the host lattice.  The
upper limit of the sum in \eqref{eq:lunnon-g} reflects the fact that
an $n$-cell polycube cannot span more than $n-1$ dimensions.

\begin{figure}
\centering
\psset{unit=0.04\columnwidth}
\begin{pspicture}(-1,-2.5)(24.5,3)
\psset{RotZ=45}
\pstThreeDPut(0,0,0){ \cube{0,0,-1} \cube{0,0,0} \cube{0,0,1}  \cube{0,0,2}  }
\rput(0,-2){$t={18(2)}$}
\pstThreeDPut(0,3.2,0){ \cube{0,0,0} \cube{0,0,1}  \cube{0,0,2} \cube{0,1,0} }
\rput(3.7,-2){$t={17}(9)$}
\pstThreeDPut(0,6.7,0){ \cube{0,0,0} \cube{0,0,1}  \cube{0,0,2} \cube{0,1,1} }
\rput(7.2,-2){$t={16}(8)$}
\pstThreeDPut(0,10.2,0){ \cube{0,1,0} \cube{0,0,1}  \cube{0,0,2} \cube{0,1,1} }
\rput(10.7,-2){$t={16}(8)$}
\pstThreeDPut(0,14.5,0){ \cube{0,0,0} \cube{0,1,0}  \cube{1,0,0} \cube{1,1,0} }
\rput(14.5,-2){$t={16}(8)$}
\pstThreeDPut(0,18.5,0){ \cube{0,0,0} \cube{0,1,0}  \cube{1,0,0} \cube{0,0,1} }
\rput(18.5,-2){$t={15}(15)$}
\pstThreeDPut(0,22.3,0){  \cube{0,1,0}  \cube{1,0,0} \cube{1,1,0} \cube{0,1,1}}
\rput(22.3,-2){$t={16(16)}$}
\end{pspicture}
\caption{Polycubes of size $n=4$ in $d=3$ dimensions and their perimeter $t$.
 The leftmost polycube is proper in $1$ dimension, the two reighmost polycubes are proper in $3$
  dimensions. All other polycubes are proper in $2$
  dimensions. Numbers in parentheses denote the perimeter of the
  polycube embedded in its proper dimension.
\label{fig:polycubes}}
\end{figure}
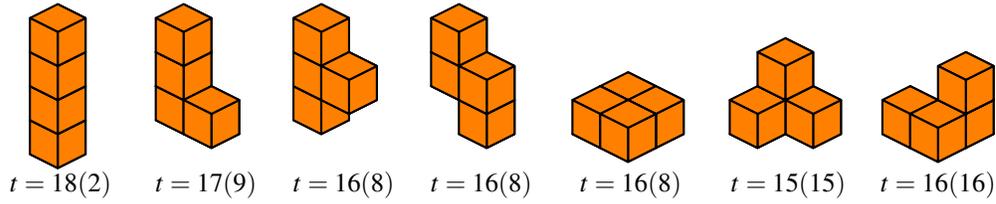

Equation~\eqref{eq:lunnon-g} generalizes a classical formula
due to Lunnon \cite{lunnon:75}, which relates the total number of polycubes
and the total number of proper polycubes of a given size disregarding their
perimeter. Equation~\eqref{eq:lunnon-g} tells us that for given $n$ and $t$, $g_{n,t}^{(d)}$ is a
polynomial in $d$. We can compute this polynomial if we know
the numbers $G_{n,t}^{(1)},\ldots,G_{n,t}^{(n-1)}$. For $n=4$, for
example, one can enumerate the proper polycubes with pencil and paper
to get 
\begin{displaymath}
  G_{4,2}^{(1)} = 1 \qquad G_{4,8}^{(2)} = 9 \qquad G_{4,9}^{(2)} = 8
  \qquad  G_{4,15}^{(3)} = 8 \qquad G_{4,16}^{(3)} = 24\,,
\end{displaymath}
and $G_{4,t}^{(i)}=0$ for all other values of $i$ and $t$. With these
numbers, \eqref{eq:lunnon-g} then yields
\begin{displaymath}
  \fl g_{4,t}^{(d)} = d \delta_{t,8d-6} + \frac{d(d-1)}{2}\left[ 9
    \delta_{t,8d-8} + 8 \delta_{t,8d-7}\right]
  +\frac{d(d-1)(d-2)}{3}\left[9 \delta_{t,8d-9} + 8 \delta_{t,8d-8} \right]\,,
\end{displaymath}
where $\delta_{x,y}$ is the usual Kronecker symbol.

The numbers $G_{n,t}^{(i)}$ can be found by enumerating polycubes on a
computer. Computer based enumerations of polycubes have a long
tradition in computer science and in statistical physics, see
\cite{luther:mertens:11a} and references therein. This approach,
however, is limited by the exponential growth of the number of proper
polycubes with $n$. The hardest problems in terms of
exhaustive enumeration are the ``diagonal'' and ``sub-diagonal"
numbers $G_{n,t}^{(n-1)}$ and $G_{n,t}^{(n-2)}$, but it turns out that
one can compute these numbers using combinatorial arguments instead of
exhaustive enumerations. This is the main result of this paper.

We start by reviewing the link between proper polycubes and edge
labelled trees. In Section~\ref{sec:G1} we demonstrate how this
link can be explored to compute the numbers $G_{n,t}^{(n-1)}$.
The computation of $G_{n,t}^{(n-2)}$ is more involved and requires
some preparation. We start by deriving formulas for the perimeter
of polycubes that correspond to certain classes of edge labeled trees
in Section~\ref{sec:perimeter-formulas}. Then we introduce a sort of
Pruefer code for edge labeled trees (Section~\ref{sec:pruefer}) that
finally allows us to compute the number of edge lebeled trees that
enter the computation of $G_{n,t}^{(n-2)}$. 
In Section~\ref{sec:application} we demonstrate how the combinatorial
arguments complement exhaustive enumerations by computing the
perimeter polynomial $P_{12}^{(d)}(q)$ for arbitrary dimension $d$. 


\section{Proper Polycubes and Tainted Trees}
\label{sec:polycubes-and-trees}

Following Lunnon \cite{lunnon:75}, we denote the total number of
polycubes of size $n$ that are proper in dimension $i$ by $\DX(n,i)$.
Obviously 
\begin{equation}
  \label{eq:1}
  \DX(n,i) = \sum_t G_{n,t}^{(i)}\,.
\end{equation}
A polycube of size $n$ can barely span $n-1$ dimensions. In particular,
in a polycube that contributes to $\DX(n,n-1)$, every pair of adjacent
cells must span a new dimension. There can't be any
loops. Apparently the computation of $\DX(n,n-1)$ is an exercise in
counting trees. Let's work out that exercise.

Consider the adjacency graph of a polycube, i.e, the graph in which
each vertex represents a cell of the polycube and two vertices are
connected if the corresponding cells are neighbors in the polycube.
The edges of this graph are labeled with the dimension along which the two
cells touch each other.

In the case $\DX(n,n-1)$ every edge in the adjacency graph has a
unique label that corresponds to one of the $n-1$ dimensions. Hence
the adjacency graph has exactly $n-1$ edges, i.e., it is an \emph{edge
labeled tree}. And since there are two possible directions for each
dimension, we get
\begin{displaymath}
  \DX(n,n-1) = 2^{n-1} \cdot\,\mbox{number of edge labeled trees of
    size $n$.}
\end{displaymath}

The number of \emph{vertex} labeled trees of size $n$ is given
by $n^{n-2}$, the famous fomula published by Cayley in 1889
\cite{cayley:1889}.  The number of \emph{edge} labeled trees
is smaller by a factor $1/n$, i.e., it is $n^{n-3}$. The earliest
reference for a proof we could find is from 1995 \cite{cameron:95}, but
the formula seems to have been known for much longer. We will give 
another proof in Section~\ref{sec:pruefer}.
Anyway, the number of proper polycubes of size $n$ in dimension $n-1$ reads
\begin{equation}
  \label{eq:dx1}
  \DX(n,n-1) = 2^{n-1}\,n^{n-3}\,.
\end{equation}
This formula has been known and used in the statistical physics community for a
long time \cite{fisher:essam:61}, but a formal proof has been
published only recently \cite{barequet:barequet:rote:10}.

The computation of $\DX(n,n-k)$ for $k>1$ is more complicated because
the correspondence between proper polycubes and
edge labeled trees gets more involved. For $k>1$, the adjacency graph
of a proper polycube can have loops, i.e., it is represented by more
than one spanning tree. On the other hand some labeled spanning trees
represent impossible polycubes with overlapping cells. Yet a careful
consideration of all these cases has yield the formulas for
$\DX(n,n-2)$ \cite{barequet:barequet:rote:10} and $\DX(n,n-3)$
\cite{asinowski:etal:11}. Using non-rigorous arguments, $\DX(n,n-k)$
has been computed up to $k=7$ \cite{luther:mertens:11a}. 
In this contribution we will show how combinatorial arguments
can be used to compute $G_{n,t}^{(n-1)}$ and $G_{n,t}^{(n-2)}$.

\section{Computation of $G_{n,t}^{(n-1)}$.}
\label{sec:G1}

For $G_{n,t}^{(n-1)}$ we need to know 
the perimeter of a polycube represented by a given tree. It turns out
that the perimeter is determined by the degree sequence of the tree:
\begin{thm} \label{thm:perimeter}
  Let $\delta=(\delta_1,\ldots,\delta_n)$ denote the degree sequence of the
  adjacency graph of a proper $n$-cell polycube in $(n-1)$
  dimensions. The perimeter $t_1$ of the polycube depends only on $n$
  and $\delta$ and it is given by
  \begin{equation}
    \label{eq:t-d}
    t_1(\delta) = (2n-1)(n-1)-\frac{1}{2}\sum_{i=1}^n \delta_i^2\,.
  \end{equation}
\end{thm}
\begin{proof}
  If we ignore the fact that some perimeter sites can be shared by
  several cells of the polycube, the perimeter of an $n$-cell polycube 
  in $d$ dimensions would be
  \begin{equation}
    \label{eq:def-tstar}
    t_{n,d}^\star = \sum_{i=1}^n (2d-\delta_i) = 2dn - 2(n-1)\,.  
  \end{equation}
  Now consider a cell of the polycube with degree $\delta_i > 1$. 
  If the polycube is proper in $(n-1)$ dimensions, each
  pair of adjacent cells spans a fresh two dimensional plane. The perimeter
  site that sits in the corner of that plane is shared by two cells of
  the polycube, and these are the only perimeter sites have been overcounted in
  $t_{n,n-1}^\star$. The true perimeter then reads  
  \begin{displaymath}
    t_1(\delta) = t_{n,n-1}^\star - \sum_{i=1}^n {\delta_i \choose 2} =
    (2n-1)(n-1)-\frac{1}{2}\sum_{i=1}^n \delta_i^2\,.
  \end{displaymath}
\end{proof}
We also need the number of edge labeled trees with a given degree
sequence. This number is given by
\begin{equation}
  \label{eq:T}
  T(\delta) = \frac{1}{n}\,\cdot\,\big(\alpha_1(\delta), \ldots,
  \alpha_{n-1}(\delta)\big)! \,\cdot\, \big(\delta_1-1, \ldots, \delta_n-1\big)!
\end{equation}
where $\delta$ represents an \emph{ordered} degree sequence and
$\alpha_k(\delta)$ denotes the number of vertices with degree $k$.
The notation $(n_1,\ldots,n_k)!$ refers to the multinomial coefficient
\begin{equation}
  \label{eq:multinomial}
  (n_1,n_2,\ldots,n_k)! = 
    \frac{(n_1+n_2+\cdots+n_k)!}{n_1!\, n_2!\,\cdots\,n_k!} \quad \text{for
      all $n_i\geq 0$.}
\end{equation}
For later convenience we define $(n_1,\ldots,n_k)!$ to be zero if any of 
the $n_i$'s is negative. Eq.~\eqref{eq:T} is a well known result (see \cite{cameron:95}, for
example). We will rederive it in Section~\ref{sec:pruefer}.
 
\begin{figure}
\begin{algorithm}[H]
   G1($n$)

     \KwIn{size $n$ of polycubes}
      \KwOut{numbers $G_{n,t}^{(n-1)}$ of proper polycubes in
        dimension $n-1$ for all perimeters $t$ }
   \Begin{
   $G_{n,t}^{(n-1)} := 0$ \;
   \ForEach{\textrm{sorted degree sequence $\delta=(\delta_1,\ldots,\delta_n)$}}{
     $t := (2n-1)(n-1) - \frac{1}{2}\sum_i \delta_i^2$ \;
     $G := 2^{n-1}\,n^{-1}\,\big(\alpha_1(\delta), \ldots,
     \alpha_{n-1}(\delta)\big)!\, \big(\delta_1-1, \ldots,
     \delta_n-1\big)!$ \;
     $G_{n,t}^{(n-1)} := G_{n,t}^{(n-1)} + G$ \;
   }
   \Return $G_{n,t}^{(n-1)}$\;
   }
\end{algorithm}
\caption{This algorithm computes the number of polycubes of
  size $n$ that are proper in $n-1$ dimensions for all possible perimeters $t$.}
\label{alg:G1}
\end{figure}

Equipped with \eqref{eq:t-d} and \eqref{eq:T} we can easily compute
$G_{n,t}^{(n-1)}$ for all possible values of $t$ by looping over all
ordered degree sequences of trees of size $n$, see Fig.~\ref{alg:G1}.
The identity $\sum\delta_i = 2|E|=2(n-1)$ tells us that the ordered sequences
$\delta-1$ correspond to the integer partitions of $n-2$. Hence the
number of iterations in our algorithm (Fig.~\ref{alg:G1}) equals the
number of integer partitions of 
$n-2$. There is no closed formula for the number of partitions, but asymptotically the
number of partitions of $n$ scales like \cite{hardy:ramanujan:18} 
\begin{equation}
  \label{eq:2}
  \sim \frac{1}{4n\sqrt{3}} \, \mathrm{e}^{\pi\sqrt{2n/3}}\,.
\end{equation}
Despite this exponential complexity, the algorithm in Fig.~\ref{alg:G1}
is much faster than the equivalent exhaustive enumeration of proper animals.  For
$n=12$, for example, there are only $42$ ordered degree sequences and
the computation of $G_{12,t}^{(11)}$ takes only a tiny
fraction of a second on a laptop. The equivalent explicit enumeration
of lattice animals would take about 25 years of CPU time, see Section~\ref{sec:application}.
years \cite{luther:mertens:11a}. A simple Python script that computes
$G_{n,t}^{(n-1)}$ can be found on the project webpage
\cite{animals:website}.

\section{Perimeter Formulas for $G_{n,t}^{(n-2)}$.}
\label{sec:perimeter-formulas}

Computing the perimeter values for $G_{n,t}^{(n-2)}$ is a bit more
complicated. Again we represent a polycube by edge labeled tree with
$n$ nodes and, in this case, $n-2$ distinct edge labels. In such a
polycube all but one connection between two adjacent cells span a new
dimension.

\begin{table}
\begin{center}
\begin{tabular}{m{0.25in}|m{1.0in}|m{2.7in}|m{0.6in}}
Type & Pattern & Perimeter and Description & Count \\
\hline
\hline
$xx$ & 
\psset{unit=0.04\columnwidth}
\begin{pspicture}(-0.5,-0.5)(4.5,2)

\psdot(0,0)
\psdot(2,0)
\psdot(4,0)
\psline[arrowscale=2.5,linewidth=0.04,ArrowInside=-<,ArrowInsidePos=0.35,ArrowFill=true](0,0)(2,0)
\psline[arrowscale=2,linewidth=0.04,ArrowInside=->,ArrowInsidePos=0.8,ArrowFill=false](0,0)(2,0)
\psline(2,0)(4,0)
\psline[arrowscale=2.5,linewidth=0.04,ArrowInside=-<,ArrowInsidePos=0.35,ArrowFill=true](2,0)(4,0)
\psline[arrowscale=2,linewidth=0.04,ArrowInside=->,ArrowInsidePos=0.8,ArrowFill=false](2,0)(4,0)

\rput(1,0.8){{$\large x$}}
\rput(3,0.8){{$\large x$}}
\end{pspicture}

& $t_{xx}(\delta) = t_2(\delta) + 1$ 
\newline
Antiparallel orientation of $x$ is no legal polycube.
& $2^{n-3}\, T_{xx}(\delta)$\\
\hline

$xyx$ & 
\psset{unit=0.04\columnwidth}
\begin{pspicture}(-1,-0.5)(4,4.5)

\psdot(0,0)
\psdot(3,0)
\psdot(0,3)
\psdot(3,3)

\psline[arrowscale=2.5,linewidth=0.04,ArrowInside=-<,ArrowInsidePos=0.35,ArrowFill=true](0,0)(0,3)
\psline[arrowscale=2,linewidth=0.04,ArrowInside=->,ArrowInsidePos=0.8,ArrowFill=false](0,0)(0,3)

\psline[arrowscale=2.5,linewidth=0.04,ArrowInside=-<,ArrowInsidePos=0.35,ArrowFill=true](3,0)(3,3)
\psline[arrowscale=2,linewidth=0.04,ArrowInside=->,ArrowInsidePos=0.8,ArrowFill=false](3,0)(3,3)

\psline(0,3)(3,3)
\psline[linestyle=dashed](0,0)(3,0)

\rput(-0.5,1.5){{$\large x$}}
\rput(3.5,1.5){{$\large x$}}
\rput(1.5,3.5){{$\large y$}}

\end{pspicture}

& $t_{xyx}(\delta, \delta_a, \delta_b) = t_2(\delta) - (\delta_a - 1)- (\delta_b - 1)$
\newline
Corresponding polycubes have $4$ spanning trees.
& $2^{n-3}\, T_{xyx}(\delta)$
\\
\hline
$xyzx$ & 
\psset{unit=0.04\columnwidth}
\begin{pspicture}(-0.3,0)(5,5)
\psdot(0,0)
\psdot(3,0)
\psdot(4.06066017177982,1.06066017177982)
\psdot[dotstyle=asterisk](1.06066017177982,1.06066017177982)

\psdot[dotstyle=asterisk](0,3)
\psdot(4.06066017177982,4.06066017177982)
\psdot(1.06066017177982,4.06066017177982)

\psline[arrowscale=2.5,linewidth=0.04,ArrowInside=-<,ArrowInsidePos=0.35,ArrowFill=true](0,0)(3,0)
\psline[arrowscale=2,linewidth=0.04,ArrowInside=->,ArrowInsidePos=0.8,ArrowFill=false](0,0)(3,0)

\psline[arrowscale=2.5,linewidth=0.04,ArrowInside=-<,ArrowInsidePos=0.35,ArrowFill=true](1.06066017177982,4.06066017177982)(4.06066017177982,4.06066017177982)
\psline[arrowscale=2,linewidth=0.04,ArrowInside=->,ArrowInsidePos=0.8,ArrowFill=false](1.06066017177982,4.06066017177982)(4.06066017177982,4.06066017177982)

\psline(3,0)(4.06066017177982,1.06066017177982)
\psline(4.06066017177982,1.06066017177982)(4.06066017177982,4.06066017177982)
\psline[linestyle=dashed](0,0)(0,3)
\psline[linestyle=dashed](0,0)(1.06066017177982,1.06066017177982)
\psline[linestyle=dashed](1.06066017177982,1.06066017177982)(1.06066017177982,4.06066017177982)
\psline[linestyle=dashed](0,3)(1.06066017177982,4.06066017177982)
\psline[linestyle=dashed](1.06066017177982,1.06066017177982)(4.06066017177982,1.06066017177982)

\rput(1.5,-0.5){{$\large x$}}
\rput(2.56066017177982,4.56066017177982){{$\large x$}}
\rput(4, 0.3){{$\large y$}}
\rput(4.56066017177982,2.56066017177982){{$\large z$}}

\end{pspicture}

& $t_{xyzx}(\delta) = t_2(\delta) - 1$ & $2^{n-3}\,  T_{xyzx}(\delta)$
\end{tabular} 
\end{center}
\caption{Counting polycubes of size $n$ that are proper in $n-2$
  dimensions. The three patterns that lead to over counting and/or a
  perimeter value that is different form $t_2(\delta)$.  
  $T_{xx}$, $T_{xyx}$ and $T_{xyzx}$  denote the number of undirected,
  edge labeled trees with the corresponding error pattern. Every tree that does not
  contain any of these patterns represents exactly $2^{n-2}$ polycubes
  with perimeter $t_2(\delta)$. Note that flipping the direction of both
  edges with the same label $x$ at the same time yields a polycube that
  must not be counted since it is only the mirror image in the
  $x$-dimension. Hence the factor $2^{n-2}$ for the error free
  case. The factor $2^{n-3}$ in the error classes arises from the fact
that an antiparallel orientiation of the $x$-edge either doesn't
correspond to a legal polycube (case $xx$) or induces no correction to
the default perimeter $t_2(\delta)$ (cases $xyz$ and $xyzx$).}
\label{table:patterns}
\end{table}

In contrast to the previous case there is no longer a one-to-one map
between polycubes and edge labeled trees. We need to take care of the three 
``error pattern'' shown in Table~\ref{table:patterns}.  

Consider the case where two adjacent edges of the tree share the same
label. We call this configuration $xx$. If the directions of these
edges are antiparallel, the corresponding polycube will have
overlapping cells. If their directions are parallel, the polycube is
valid but has a larger perimeter than the error-free case.

An $n$-cell polycube that is proper in $n-2$ dimensions can have a
loop, in which case it is represented by more than one spanning
tree. The loopy polycubes all contain a quadrilateral. The
corresponding tree then contains the pattern $xyx$ of adjacent edge
labels.  Only one fourth of
those trees contribute to the number of polycubes, and their perimeter
is also special.

Taking into account the patterns $xx$ and $xyx$ is enough to correctly
compute $DX(n,n-2)$, but to compute
$G_{n,t}^{(n-2)}$ an additional pattern has to be considered: a
pattern $xyzx$ of adjacent edges represents exactly one polycube, but its
perimeter is different from the perimeter of other trees with the
same degree sequence. But let's start with the most simple, error free case:

\begin{thm} \label{thm:perimeter2}
Let $\delta=(\delta_1,\ldots,\delta_n)$ denote the degree sequence of an directed edge
  labeled tree with $n-2$ distinct edge labels not containing any of
  the error patterns in Table \ref{table:patterns}.
  Then the perimeter of the corresponding polycube is
  \begin{equation}
    \label{eq:t-xx}
    t_2(\delta) = (2n^2 - 5n+ 1) - \frac 12 \sum_{i=1}^n \delta_i^2\,.
  \end{equation}
\end{thm}
\begin{proof}
  As with $t_1$ we start with $t_{n,d}^\star$ and subtract the number of
  those perimeter cells that are shared by two cells. Since we excluded all error patterns,
  the two equally labeled edges are far apart from each other. This means that the tree
  locally looks like a tree with $n-1$ distinct edge labels. Consequently all pairs of cells
  that are adajcent to a common cell span a fresh two dimensional plane and hence share a 
  perimeter cell:
  \begin{displaymath}
    t_2(\delta) = t_{n,n-2}^\star - \sum_{i=1}^n {\delta_i \choose 2} =
    (2n^2 - 5n+ 1) - \frac 12 \sum_{i=1}^n \delta_i^2\,.
  \end{displaymath}
\end{proof}

\begin{thm} \label{thm:perimeter_xx}
  Let $\delta=(\delta_1,\ldots,\delta_n)$ denote the degree sequence of an directed edge
  labeled tree with $n-2$ distinct edge labels containing the pattern $xx$. Then the perimeter
  of the corresponding polycube is
  \begin{equation}
    \label{eq:t-xx}
    t_{xx}(\delta) = t_2(\delta) + 1\,.
  \end{equation}
\end{thm}
\begin{proof}
  The same argument as in Theorem~\ref{thm:perimeter2} applies to all pairs of edges
  that are adjacent to a common cell, except the two equally labeled edges. They don't
  span a plane but form a one dimensional line, hence we have to add
  back the perimeter cell that has been mistakenly subtracted in $t_2(\delta)$.
\end{proof}

\begin{thm} \label{thm:perimeter_xyx}
  Let $\delta=(\delta_1,\ldots,\delta_n)$ denote the degree sequence of an directed edge
  labeled tree with $n-2$ distinct edge labels containing the pattern $xyx$. Let $\delta_a$ and $\delta_b$ be the degrees
  of the two vertices that lie on the ends of the path formed by the two equllay labeled
  edges and the edge separating them. Then the perimeter of the corresponding polycube
  is
  \begin{equation}
    \label{eq:t-xx}
    t_{xyx}(\delta) = t_2(\delta) - \left(\delta_a - 1\right) - \left(\delta_b - 1\right)
  \end{equation}
\end{thm}
\begin{proof}
  Polycubes containg the quadrilateral have four spanning trees. To construct a spanning
  tree for a given polycube, one of the four edges of the quadrilateral has to be removed.
  The missing edge would connect the cells $a$ and $b$. Since the edge is missing from the tree
  there are more pairs of cells that are neighbours than we count in $t_2$. In $t_2$ we subtract
  \begin{displaymath}
    {\delta_a \choose 2 } + {\delta_b \choose 2 }
  \end{displaymath}
  pairs of cells, but if the edge were present we would have subtracted 
  \begin{displaymath}
    {\delta_a + 1\choose 2 } + {\delta_b + 1\choose 2 }
  \end{displaymath}
  Unfortunately this is still not correct, because two of those pairs span the plane of
  the quadrilateral and must not be subtracted,
  \begin{displaymath}
    \eqalign{
    t_{xyx}(\delta) &= t_2(\delta) + {\delta_a \choose 2 } + {\delta_b \choose 2 } - \left[{\delta_a + 1\choose 2 } +
{\delta_b + 1\choose 2 } - 2\right] \\
    &=  t_2(\delta) - \left(\delta_a -
1\right) - \left(\delta_b - 1\right)\,.
    }  
  \end{displaymath}
\end{proof}

\begin{thm} \label{thm:perimeter_xyzx}
  Let $\delta=(\delta_1,\ldots,\delta_n)$ denote the degree sequence of an directed edge
  labeled tree with $n-2$ distinct edge labels containg the pattern $xyzx$. Then the perimeter of
  the corresponding poylcube is
  \begin{equation}
    \label{eq:txyzx-d}
    t_{xyzx}(\delta) = t_2(\delta) - 1\,.
  \end{equation}
\end{thm}
\begin{proof}
  The two end vertices of the path connecting the two equally labeled edges have two common neighbouring
  perimeter cells. One of them is accounted for correctly. It is adajcent to the two end points and to another point
  in the polycube and those counted three times. On the other hand it lies in two planes formed by the cells
  of the pattern $xyzx$, which means it is subtracted twice. Only the double counting of the other perimeter cell
  is missed and must be corrected.
\end{proof}

\section{Prüfer-like bijection for edge labeled trees}
\label{sec:pruefer}

In order to characterize and count the edge labeled trees with the patterns
listed in Table~\ref{table:patterns}, we will use a bijection
between edge labeled trees of size $n$ and sequences of integers $0,\dots,n-1$
of length $n-3$ \cite{pikhurko:05}, which  
is very similar to the Prüfer code for vertex labeled trees. En
passant we will prove equations \eqref{eq:dx1} and \eqref{eq:T}.

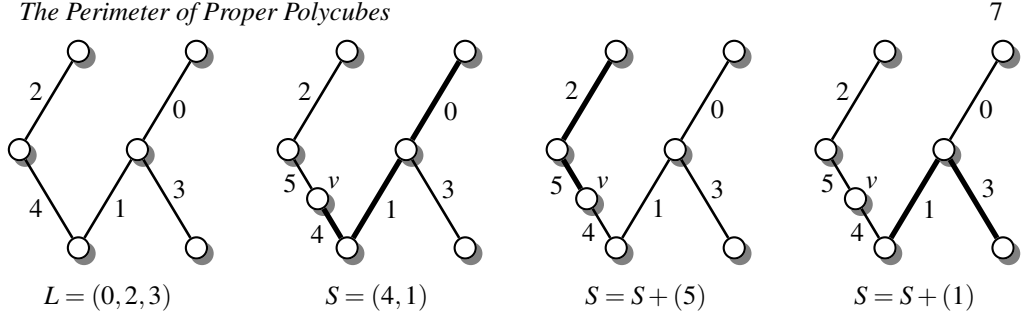
\begin{figure}
  \centering
  \psset{xunit=0.06\textwidth,yunit=0.1\textwidth,labelsep=1pt}
  \begin{pspicture}(0,0.5)(3,3)
     \graphnodes
     \Cnode(1,3){n1} \Cnode(0,2){n2}
     \Cnode(3,3){n3} \Cnode(2,2){n4}
     \Cnode(1,1){n5} \Cnode(3,1){n6}
     \psset{shadow=false,linewidth=1pt}
     \ncline{n1}{n2}\nbput{$2$}
     \ncline{n2}{n5}\nbput{$4$}
     \ncline{n5}{n4}\nbput{$1$}
     \ncline{n6}{n4}\nbput{$3$}
     \ncline{n4}{n3}\nbput{$0$}
     \rput(1.5,0.5){$L = (0,2,3)$}
  \end{pspicture}
  \hfill
  \begin{pspicture}(0,0.5)(3,3)
     \graphnodes
     \Cnode(1,3){n1} \Cnode(0,2){n2}
     \Cnode(3,3){n3} \Cnode(2,2){n4}
     \Cnode(1,1){n5} \Cnode(3,1){n6}
     \Cnode(0.5,1.5){n7} \nput{45}{n7}{$v$}
     \psset{shadow=false,linewidth=1pt}
     \ncline{n1}{n2}\nbput{$2$}
     \ncline{n2}{n7}\nbput{$5$}
    \ncline{n6}{n4}\nbput{$3$}
    \psset{linewidth=2pt}
     \ncline{n7}{n5}\nbput{$4$}
     \ncline{n5}{n4}\nbput{$1$}
     \ncline{n4}{n3}\nbput{$0$}
     \rput(1.5,0.5){$S = (4,1)$}
  \end{pspicture}
  \hfill
  \begin{pspicture}(0,0.5)(3,3)
     \graphnodes
     \Cnode(1,3){n1} \Cnode(0,2){n2}
     \Cnode(3,3){n3} \Cnode(2,2){n4}
     \Cnode(1,1){n5} \Cnode(3,1){n6}
     \Cnode(0.5,1.5){n7}\nput{45}{n7}{$v$}
     \psset{shadow=false,linewidth=1pt}
    \ncline{n6}{n4}\nbput{$3$}
     \ncline{n7}{n5}\nbput{$4$}
     \ncline{n5}{n4}\nbput{$1$}
     \ncline{n4}{n3}\nbput{$0$}
      \psset{linewidth=2pt}
     \ncline{n1}{n2}\nbput{$2$}
     \ncline{n2}{n7}\nbput{$5$}
     \rput(1.5,0.5){$S = S + (5)$}
  \end{pspicture}
   \hfill
  \begin{pspicture}(0,0.5)(3,3)
     \graphnodes
     \Cnode(1,3){n1} \Cnode(0,2){n2}
     \Cnode(3,3){n3} \Cnode(2,2){n4}
     \Cnode(1,1){n5} \Cnode(3,1){n6}
     \Cnode(0.5,1.5){n7}\nput{45}{n7}{$v$}
     \psset{shadow=false,linewidth=1pt}
     \ncline{n7}{n5}\nbput{$4$}
     \ncline{n4}{n3}\nbput{$0$}
     \ncline{n1}{n2}\nbput{$2$}
     \ncline{n2}{n7}\nbput{$5$}
     \psset{linewidth=2pt}
     \ncline{n5}{n4}\nbput{$1$}
     \ncline{n6}{n4}\nbput{$3$}
     \rput(1.5,0.5){$S = S + (1)$}
  \end{pspicture}
  \caption{An example for the Prüfer code for edge labeled trees. The
    code for this tree reads $C=(1,5,1)$.} 
  \label{fig:pruefer}
\end{figure}
 
Let $T$ be a tree with $n$ vertices and edge labels
$0,1,\dots,n-2$, and let $L$ be the sorted list
of labeled edges that are incident to a leaf.
We subdivide the edge ${n-2}$ into two edges by adding a new vertex $v$. 
The new edge which lies on the unique path from $v$ to
the smallest element in $L$ inherits the old label $n-2$. The
other new edge gets the new label $n-1$. See Fig.~\ref{fig:pruefer} 
for an illustration of this step and the rest of the procedure.

We now construct a sequence $S$ of edge labels. Starting from the empty
sequence, repeat the following procedure for each $i\in L$ in
increasing order. Let $i,e_1,e_2,\ldots$ denote the path from $i$
to $v$.  Move along this path and stop when you encounter an edge $e_i
\in\{{n-1},{n-2},S\}$. Append the sequence $(e_i,e_{i-1},\dots,e_1)$
to $S$. 

Note that $S$ always starts with ${n-2}$ and has length $n-2$. We
form the sequence $C(T)$ from $S$ by removing the first entry
${n-2}$. This is the code for the edge labeled tree $T$.

For the reverse construction consider a sequence $C$ of length $n-3$
consisting of integers from the set $\{0,\dots,{n-1}\}$. We obtain a
new sequence $S$ from $C$ by adding a leading ${n-2}$. We define $L=(\ell_1,\ell_2,\dots)$
as the list of all labels not contained in $S$, sorted in increasing
order $\ell_1 < \ell_2 < \dots$. Note that $L$ contains at least two
elements. Obviously an element of $S$ equals ${n-2}$, ${n-1}$ or
some previously occuring element of $S$ exactly $|L|$ times. Cutting
$S$ before each such element gets us subsequences
$S_1,\dots,S_{|L|}$. Note that each $S_i$ begins with ${n-1}$,
${n-2}$ or a label contained in some previous $S_j$, $j<i$.
 
We append $\ell_i$ to $S_i$ to create the
sequence $P_i$. Each $P_i$ represents a path from the interior of the tree
to one of its leafs, and we will use these paths to assemble the tree.

We start with $T$ being the path consisting of ${n-1}$ and
${n-2}$ adjacent to the vertex $v$. For $i=1,\dots,|L|$ we append
the path $P_i$ to $T$ along its unique common edge $e$. Of the two
possible ways for doing this, we choose the one for which the
path connecting $v$ and $\ell_i$ uses the edge $e$.

Finally we remove the vertex $v$ and merge the two edges ${n-1}$ and
${n-2}$ to a single edge with label ${n-2}$. 

This concludes the proof that there is a one-to-one correspondence
between edge-labeled trees of size $n$ and sequences of length $n-3$
in which each element is a number between $0$ and $n-1$. In particular
this proves that the number of edge labeled trees is $n^{n-3}$ and
hence \eqref{eq:dx1}.



The connection between the degree sequence $\delta$ and the code $C$
is given by the following theorem:
\begin{thm} \label{thm:delta-code}
  Let $T$ be an edge labeled tree of size $n$ with code $C$ and degree sequence
  $\delta=(\delta_1,\ldots,\delta_n)$. Then for $i=1,\ldots,n$ the
  sequence $S=(n-2,C)$ contains a label $l_i$ exactly $\delta_i-1$ times and vice versa.
\end{thm}
\begin{proof}
  Let $v_i$ be a vertex with degree $\delta_i$. One of its $\delta_i$
  edges lies in the unique path from $v_i$ to the extra vertex
  $v$. Let $l_i$ be the label of that edge. Then there are exactly
  $\delta_i-1$ paths that go through $v_i$ and $l_i$ to reach
  $v$. Hence $S$ contains the label $l_i$ exactly $\delta_i - 1$
  times.
\end{proof}
Theorem~\ref{thm:delta-code} allows us to derive the number of edge
labeled trees with a given degree sequence, i.e., equation
\eqref{eq:T}. See Figure~\ref{fig:T-example} for an example. 
Note that the numbers in $S$ come in groups
of size $\delta_i-1$, and there are $(\delta_1-1, \ldots,
\delta_n-1)!$ different ways to arrange these group elements. This gives the
second multinomial coefficient in \eqref{eq:T}. Next we have to count
the number of ways to assign labels to the different groups in $S$.
Note that codes like $S=(a,a,b,b,b,c,c,c)$ and $S=(a,a,c,c,c,b,b,b)$
correspond to the same set of edge labeled trees. Therefore we must count
the number of maps from $\{1,\dots,n\}$ to $(a,b,c)$ with $b$ and $c$
being indistinguishable. In general, the number of groups in $S$ of
the size $d$ is given by $\alpha_d$. Hence the number of different
label assignments reads $(\alpha_1, \ldots, \alpha_{n-1})!$, which is the
first multinomial coefficient in \eqref{eq:T}. Finally the factor $1/n$
takes into account the fact that the first entry in $S$ is fixed to
the value $n-2$.

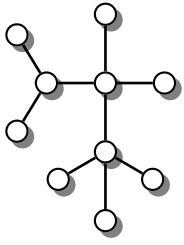
\begin{figure}
  \centering
    \psset{xunit=0.06\textwidth,yunit=0.07\textwidth,labelsep=1pt}
    \begin{pspicture}(0,-0.5)(3,3)
     \graphnodes
     \Cnode(2,0){n1} \Cnode(2,1){n3}
     \Cnode(1.2,0.6){n2} \Cnode(2.8,0.6){n4}
     \Cnode(2,2){n5} \Cnode(3,2){n7}
     \Cnode(2,3){n6} \Cnode(1,2){n8}
     \Cnode(0.5,2.7,){n9} \Cnode(0.5,1.3){n10}
     \psset{shadow=false,linewidth=1pt}
     \ncline{n1}{n3}
     \ncline{n3}{n2} \ncline{n3}{n4} \ncline{n3}{n5}
     \ncline{n5}{n6} \ncline{n5}{n7} \ncline{n5}{n8}
     \ncline{n8}{n9} \ncline{n8}{n10}
  \end{pspicture}
  \hspace{0.05\textwidth}
  \begin{minipage}[b]{0.7\textwidth}
    \begin{displaymath}
      \eqalign{
      \delta &= (1,1,1,1,1,1,1,3,4,4)\\
      \alpha &= (7,0,1,2,0,0,0,0,0)
      }
    \end{displaymath}
   According to Theorem~\ref{thm:delta-code} , the code $S$ must be of
   the form
    \begin{displaymath}
          S = (a,a,b,b,b,c,c,c)\,\, \text{or permutations.}
    \end{displaymath}
    There are $(2,3,3)! = 560$ such arrangements. Assignment of labels:
    $\{1,\ldots,10\} \mapsto (a,b,c)$ with $b,c$ indistinguishable.
    There are $(7,1,2)! = 10\cdot9\cdot8/2= 360$ possible assignments.
  \end{minipage}
  \caption{Example for counting the number of edge labelled trees with
    a given degree sequence using the Pruefer code. According to
    Equation~\eqref{eq:T} there are $560\cdot360/10 = 20160$ edge
    labeled trees with the given degree sequence $\delta$.}
  \label{fig:T-example}
\end{figure}

\section{Computing $G_{n,t}^{(n-2)}$}


We are now ready to compute the number of edge labeled trees
containing one of the patterns in Table \ref{table:patterns}
according to their degree sequence. The idea is to express the
patterns as restrictions on the Prüfer code $C$.

First note that two edges now carry the same label. We still use
$0,\ldots,n-2$ for the labels, but with the convention that the label
$n-2$ is identified with one of the other labels. By taking into
account a factor $n-2$, we can fix this shared label to be $0$.

Let's start with the pattern $xx$, i.e., with trees in which two adjacent
edges carry the same label. What are the codes $C$ that
correspond to trees in which the labels $0$ and $n-2$ are assigned to
adjacent edges? Recall that $C$ contains paths from leaf edges to the
edges $n-2$ or $n-1$ (which also represents the edge $n-2$ in the
original tree) or to edges already present in $C$. Hence $0$ and $n-2$
are incident if the first appearance of $0$ in $C$ is of the form
$(\ldots,n-2,0,\ldots)$ or $(\ldots,n-1,0,\ldots)$. If $0$ lies on the
very first path (from the smallest leaf to $n-2$), the ending edge
$n-2$ is removed from $S$, so another possibility is $C=(0,\ldots)$.
If $0$ is a leaf edge (i.e., $0\not\in C$), then it is the smallest
leaf edge and hence the origin of the very first path in $C$. If $0$
and $n-2$ are incident, there are no edges for the second path to
terminate other than $n-2$ or $n-1$. Hence $C=(n-2,\ldots)$ or
$C=(n-1,\ldots)$. In summary, the pattern $xx$ is represented by these
five classes of $C$:
\begin{enumerate}
  \item $C=(n-2,\dots)$, $0\not\in C$
  \item $C=(n-1,\dots)$, $0\not\in C$
  \item $C=(0,\dots)$
  \item $C=(\dots,n-2,0,\dots)$ for the first occurence of $0$ in $C$
  \item $C=(\dots,n-1,0,\dots)$ dto.
\end{enumerate}

Let $T_{xx,1}(\delta)$ denote the number of edge labeled trees
in class $1$.  We illustrate the computation of $T_{xx,1}(\delta)$
using the example from Figure~\ref{fig:T-example}. The first two
entries of $S$ must be equal, hence $S$ must be of the form 
\begin{displaymath}
  S = \left\{\begin{array}{ll}
    (a,a;\, b,b,b,c,c,c) & \text{with $(3,3)! = 20$ arrangements,
      or} \\
   (b,b;\, a,a,b,c,c,c) & \text{with $(1,2,3)! = 60$ arrangements.}
  \end{array}\right.
\end{displaymath}
Note that the part left of the semicolon is not to be rearranged and
that the vertex in which the edges $0$ and $n-2$
meet has either degree $3$ or degree $4$. In general, it can have
any degree $d\geq3$, and  
the corresponding number of arrangements is given by
\begin{equation}
  \label{eq:S-arrangements-xx1}
  (\delta_1-1, \ldots, \delta_n-1)! \,\frac{(d-2)(d-1)}{(n-2)(n-3)}\,,
\end{equation}
where the factor $(d-2)(d-1)/(n-2)(n-3)$ takes care of the fact that two
edges of some vertex with degree $d$ are no longer taking part in the
rearrangements in $S$.

We also need to count the number of ways to assign labels to the
elements of $S$. The value of the first two entries of $S$ is fixed to
$n-2$, and the value $0$ is not allowed in $S$. This leaves us
with $n-2$ possible values to be assigned to the other entries of $S$.
In our example we have
\begin{displaymath}
    S =
  \left\{\begin{array}{ll}
    (a,a;\, b,b,b,c,c,c) & \text{with $8\cdot 7/2 = 28$ label assignments, and
      or} \\
   (b,b;\, a,a,b,c,c,c) & \text{with $8\cdot 7 = 56$ asignments.}
  \end{array}\right.
\end{displaymath}
In general the number of label assigments for the class $xx,1$ is
given by
\begin{equation}
  \label{eq:assignments-xx1}
  (\alpha_1-1,\alpha_2,\ldots,\alpha_d-1,\ldots,\alpha_{n-1})! =
  (\alpha_1,\ldots,\alpha_{n-1})!\,\frac{\alpha_1 \alpha_d}{n(n-1)}
\end{equation}
if $d\geq 3$ is the degree of the vertex incident to the two edges with
the same label.
Combining \eqref{eq:S-arrangements-xx1}, \eqref{eq:assignments-xx1}
and \eqref{eq:T} 
and taking into account the factor $n-2$ from fixing the shared label
to be $0$ provides us with
\begin{equation}
  \label{eq:Txx1}
  \fl
  \eqalign{
  T_{xx,1}(\delta) &= \frac{(\delta_1-1, \ldots, \delta_n-1)!}{n-3}
  \sum_{d=3}^{n-1} (d-2)(d-1)\,(\alpha_1-1,
  \alpha_2,\ldots,\alpha_d-1,\ldots,\alpha_{n-1})!\\
  &= T(\delta)\,\frac{\alpha_1}{(n-1)(n-3)}\sum_{d=3}^{n-1}
  (d-2)(d-1)\,\alpha_d\,.
  }
\end{equation}

The second case is similar except for the fact that we need to
consider two degrees $d$ and $d'$. As above, $d$ denotes the
degree of the vertex that joins edges $0$ and $n-2$. The degree
$d'$ denotes the degree of the vertex that connects the edge $n-1$
to its first path in $S$. The number of arrangements of labels in $S$
that is comaptible with $S=(n-2,n-1,\ldots)$ then reads
\begin{equation}
  \label{eq:S-arrangements-xx2}
  (\delta_1-1,\ldots,\delta_n-1)! \cdot \frac{(d-1)(d'-1)}{(n-2)(n-3)}\,,
\end{equation}
and the number of label assigments is given by
\begin{equation}
  \label{eq:assignments-xx2}
  (\alpha_1,\ldots,\alpha_{n-1})!\cdot\frac{\alpha_1}{n(n-1)(n-2)}\cdot
  \left\{\begin{array}{ll}
    \alpha_d \alpha_{d'} & \text{if $d\neq d'$,} \\
    \alpha_d(\alpha_d-1) & \text{if $d=d'$.}
  \end{array}\right.
\end{equation}
Combining \eqref{eq:S-arrangements-xx2} and \eqref{eq:assignments-xx2}
and using the identity
\begin{equation}
  \label{eq:sum-d-alphad}
  \sum_{d=1}^{n-1} (d-1)\alpha_d = n-2
\end{equation}
provides us with
\begin{equation}
  \label{eq:Txx2}
  T_{xx,2}(\delta) =
  T(\delta)\,\frac{\alpha_1}{(n-1)(n-2)(n-3)}\,\left[(n-2)^2 - \sum_{d=2}^{n-1}(d-1)^2\alpha_d\right]\,.
\end{equation}

For the third case $S=(n-2,0,\ldots)$ we need to consider the two degrees
$d$ and $d'$ of the vertices incident to the edge $0$. The number of
arrangements in $S$ is again given by
\eqref{eq:S-arrangements-xx2}. We no longer have the constraint
$0\not\in S$, hence the number of label assignments is
\begin{equation}
  \label{eq:assignments-xx3}
  (\alpha_1,\ldots,\alpha_{n-1})!\cdot\frac{1}{n(n-1)}\cdot
  \left\{\begin{array}{ll}
    \alpha_d \alpha_{d'} & \text{if $d\neq d'$,} \\
    \alpha_d(\alpha_d-1) & \text{if $d=d'$,}
  \end{array}\right.
\end{equation}
which gets us
\begin{equation}
  \label{eq:Txx3}
  T_{xx,3}(\delta) =
  T(\delta)\,\frac{1}{(n-1)(n-3)}\,\left[(n-2)^2 - \sum_{d=2}^{n-1}(d-1)^2\alpha_d\right]\,.
\end{equation}

For the fourth class of $xx$-patterns we again consider the degrees $d$ and
$d'$ of the vertices incident to the edge with label $0$. Let $d'$
denote the degree of the vertex that is not incident to the edge
$n-2$. Then $0$ appears $d'-1$ times in $S$. If $j$ denotes the
position of the first $0$ in $S$ there are
\begin{displaymath}
  {n-2-j} \choose {d'-2}
\end{displaymath}
ways to distribute the other $0$'s to the right of position $j$.
Now $j\geq 3$ (the case $j=2$
corresponds to the previous class $xx,3$) and the total number of
arrangements of the $0$'s in $S$ reads
\begin{equation}
  \label{eq:position-count}
  \sum_{j=3}^{n-2} {{n-2-j} \choose {d'-2}} = \sum_{k=0}^{n-5} 
  {k \choose{d'-2}} = {{n-4}\choose{d'-1}}\,.
\end{equation}
Besides all the $0$'s, we have also two entries $n-2$ in $S$ with fixed
positions, both with degree $d$. Hence the total number of
arrangements in $S$ is
\begin{equation}
  \label{eq:S-arrangements-xx4}
  \fl (\delta_1-1,\ldots,\delta_n-1)! \cdot \frac{(d-1)(d-2)}{(n-2)(n-3)}
  \cdot \underbrace{\frac{(d'-1)!\,(n-4-(d'-1))!}{(n-4)!} \cdot  {{n-4}\choose{d'-1}}}_{=1}
\end{equation} 
The number of label assignments is again given by
\eqref{eq:assignments-xx3}, and combining these counts we get
\begin{equation}
  \label{eq:Txx4}
  T_{xx,4}(\delta) =
  T(\delta)\,\frac{n-1-\alpha_1}{(n-1)(n-3)}\,\sum_{d=3}^{n-1}(d-1)(d-2)\,\alpha_d\,.
\end{equation}

The fifth and final class of the error pattern $xx$ is characterized
by the degrees $d$ and $d'$ of the endpoints of edge $0$ and $d''$, the
degree of the endpoint of edge $n-2$ that is not incident to edge
$n-2$. The argument with the first position of the $0$ entry in $S$
works exactly as above and leaves us with
\begin{equation}
  \label{eq:S-arrangements-xx5}
  (\delta_1-1,\ldots,\delta_n-1)! \cdot \frac{(d-1)(d''-1)}{(n-2)(n-3)}
\end{equation} 
for the number of arrangements in $S$. The number of label assignments
reads
\begin{equation}
  \label{eq:assignments-xx5}
  \frac{(\alpha_1,\ldots,\alpha_{n-1})!}{n(n-1)(n-2)}\cdot
  \left\{\begin{array}{ll}
    \alpha_d \alpha_{d'} \alpha_{d''}& \text{if $d\neq d' \neq d''$ pairwise,} \\
    \alpha_d(\alpha_d-1)\alpha_{d'} & \text{if $d=d''$ and $d\neq
      d'$,}\\
    \alpha_d(\alpha_d-1)\alpha_{d''} & \text{if $d=d'$ and $d\neq
      d''$,}\\
    \alpha_{d'}(\alpha_{d'}-1)\alpha_{d} & \text{if $d'=d''$ and $d'\neq
      d$,}\\
    \alpha_d(\alpha_d-1)(\alpha_d-2) & \text{if $d=d'=d''$.}
  \end{array}\right.
\end{equation}
Combining these counts provides us with
\begin{equation}
  \label{eq:Txx5}
  T_{xx,5}(\delta) =
  T(\delta)\,\frac{n-2-\alpha_1}{(n-1)(n-2)(n-3)}\,\left[(n-2)^2-\sum_{d=2}^{n-1}(d-1)^2\,\alpha_d\right]\,.
\end{equation}
Now the total number of edge labeled trees in this error class is the sum
$T_{xx} = T_{xx,1} + \cdots + T_{xx,5}$, which yields the
surprisingly simple expression
\begin{equation}
  \label{eq:Txx}
  T_{xx}(\delta) = 
  (n-2) T(\delta)\,\frac{\sum_d {d \choose 2} \alpha_d}{{{n-1} \choose 2}}\,,
\end{equation}
With the benefit of hindsight, this formula could have been written
down without going through all the combinatorial arguments above:
$(n-2) T(\delta)$ is the number of edge labeled trees with $n-2$
distinguishable labels, and the other factor is the fraction of all
pairs of edges that are incident to a common vertex (with degree $d$,
sum over $d$). 

The reason we went through all the combinatorics to arrive at the
simple formula \eqref{eq:Txx} is to demonstrate the principle. It
turns out that for $T_{xyx}$ and $T_{xyzx}$, there are no such simple
formulas, and going through the various classes of constraints on $C$
seems to be the only way to compute $T_{xyx}$ and $T_{xyzx}$. The
actual computation is similar in spirit to the computation of
$T_{xx}$, but the number of classes is larger and the corresponding
formulas are more complicated. We present the results in
\ref{sec:Txyx} and \ref{sec:Txyzx}. The corresponding formulas have
been implemented in a Python script that computes $G_{n,t}^{(n-2)}$
for given $n$. This script (and the script for $G_{n,t}^{(n-1)}$) can
be found on the project webpage \cite{animals:website}.

\section{Application}
\label{sec:application}

To compute the perimeter polynomial $P_n^{(d)}(q)$ for arbitrary
dimension $d$, it suffices to know the numbers $G_{n,t}^{(i)}$ for
$i=1,\ldots,n-1$. Since it is more efficient to enumerate lattice
animals without keeping track of their spanning dimensionality,
computer enumerations usually yield $g$ rather than $G$. 
But we can compute $G$ from $g$ because \eqref{eq:lunnon-g} can be
inverted \cite{call:velleman:93}:
\begin{equation}
  \label{eq:lunnon-invers}
  G_{n,t}^{(i)} = \sum_{d=1}^{i} (-1)^{i-d} {i \choose d}\, g^{(d)}_{n,t+2(d-i)n}\,.
\end{equation}

\begin{table}
    \centering
    \begin{tabular}{crrrrrrrrr}
      $d = $ & 2 & 3 & 4 & 5 & 6 & 7 & 8 & 9 & 10 \\
      $n \leq $ & 22 & 18 & 15 & 14 & 14 & 13 & 11 & 11 & 11
    \end{tabular}
    \caption{Parameters for which enumeration data for $g_{n,t}^{(d)}$
      are available
      \cite{mertens:90,luther:mertens:11a}.}
    \label{tab:data}
\end{table}

The available enumeration data (Table~\ref{tab:data}) suffices to
compute $P_n^{(d)}(q)$ for $n\leq 11$. For $n=12$, we need additional
data for the number of lattice animals of size $12$ in dimensions $8$,
$9$, $10$ and $11$, but thanks to our combinatorial results we need
to run exhaustive enumerations only in dimensions $8$ and $9$. How
much computation time does this save us?

For fixed value of $n$, the number of polycubes to be enumerated grows
only polynomially with $d$, but we are talking about large
numbers here. Let
\begin{equation}
  \label{eq:Ad}
  A_d(n) = \sum_t g_{n,t}^{(d)}
\end{equation}
the total number of polycubes of size $n$ in dimension $d$. These
numbers are known for all values of $d$ and $n\leq 14$
\cite{luther:mertens:11a}. In particular,
\begin{displaymath}
  \eqalign{\frac{A_{10}(12)}{A_9(12)} &=
  \frac{\numprint{7412808050184625}}{\numprint{2001985304489169}} \simeq 3.7 \\
  \frac{A_{11}(12)}{A_9(12)} &=
  \frac{\numprint{23882895566952721}}{\numprint{2001985304489169}} \simeq 11.9\,.}
\end{displaymath}
Our combinatorial method saves us about $15.6$ times the CPU time of
the largest enumeration task $n=12$ and $d=9$. 
The enumeration algorithm takes about 100 clock cycles to generate a
polycube \cite{luther:mertens:11a}, which implies a counting rate of roughly \numprint{3e7}
polycubes per second on a 3 GHz CPU. The enumeration for $n=12$ and
$d=9$ takes 2.12 years on a single CPU, and our combinatorics
saves us 33 years of CPU time. In practice, the enumeration is of
course parallelized and run on hundreds of CPUs, reducing the
wallclock time from ``years'' to ``days.''  In this setup, our
combinatorics still saves us about one months of computation.

The new (and the previous) data for $g_{12,t}^{(d)}$ and $G_{12.t}^{(d)}$ for
$d=8,9,10,11$ is available on the project website \cite{animals:website}.

What about the perimeter polynomials for $n=13$ and arbitrary $d$? The
combinatorics provide the data for $d=12$ and $d=11$, but we would still
need to explicitely enumerate
\begin{displaymath}
  \eqalign{
  A_8(13) &= \phantom{8}\numprint{13228272325440164} \quad
  \phantom{2} (14\,\mbox{years})\\
  A_9(13) &= \phantom{8}\numprint{67341781440810531} \quad  \phantom{2}(71\,\mbox{years})\\
  A_{10}(13) &= \numprint{282338750889253800} \quad  (298\,\mbox{years})}
\end{displaymath}
polycubes. The corresponding estimates of single CPU times are given
in parentheses. Note that our combinatorial results here save us
about 1073 years (for $A_{11}(13)$) and 3416 years (for $A_{12}(13)$) of CPU
time, but the remaining enumeration task
is still out of reach unless one is willing to devote a substantial number
of CPUs and time to accomplish it.

\section{Conclusions}

We have derived formulas for $G_{n,t}^{(n-1)}$ and  $G_{n,t}^{(n-2)}$,
the number of lattice animals of size $n$ and perimeter $t$ that are proper in dimension
$n-1$ or $n-2$. These formulas complement data from exhaustive
enumerations. In particular, they replace the hardest
enumeration tasks by an evaluation of these formulas, which is much
faster. The approach outlined in this paper can be extended to derive
formulas for  $G_{n,t}^{(n-k)}$ for $k > 2$, but the complexity
increases considerably with $k$ since more and more ``error patterns''
need to be identified and analysed. 


\appendix

\section{The case $T_{xyz}$}
\label{sec:Txyx}

For the $xyx$ case we have:
\begin{enumerate}

\item $C=(1)(n-2)\dots$, $0\not\in C$, free variables: $d_{n-1}$
 \item $C=(1)(n-1)\dots$, $0\not\in C$, free variables: $d_{n-1}$
 \item $C=(1)(1)\dots$, $0\not\in C$, free variables: $d_{n-1}$
 
 \item $C=(1)(0)\dots$, free variables: $d_0, d_{n-1}$
 \item $C=(1)\dots(1)(0)\dots$, first $0$ at position $j$, no $0$ in the first $\dots$, free variables: $d_0, d_{n-1}$
 \item $C=\dots(n-2)(1)(0)\dots$, first $0$ at position $j$, no $0$ and $1$ in the first $\dots$, free variables: $d_0, d_{n-1}$
 \item $C=\dots(n-1)(1)(0)\dots$, first $0$ at position $j$, no $0$ and $1$ in the first $\dots$, free variables: $d_0, d_{n-2}$

 \item $C=\dots(n-2)(1)\dots(1)(0)\dots$, first $1$ at position $k$, first $0$ at position $j$, no $0$ and $1$ in the first $\dots$, no $0$ in the second $\dots$, free variables: $d_0, d_{n-1}$
 \item $C=\dots(n-1)(1)\dots(1)(0)\dots$, first $1$ at position $k$, first $0$ at position $j$, no $0$ and $1$ in the first $\dots$, no $0$ in the second $\dots$, free variables: $d_0, d_{n-2}$

\end{enumerate}

One thing that is special about the $xyx$ case is that we have to specify the degrees of the two vertices at the end of
the path that that connects the equally labeled edges. One of these two vertices is always $0$ and the other one is either $n-1$ or $n-2$. This
means we must not sum over the degree of these nodes, but leave their degree
as free argument in the $T_{xyx, i}$ formulas. This is neccesary because the perimeter of the corresponding
polycube depends on these two degrees, as explained earlier.

To shorten the formulas we introduce the following variables:
\begin{equation}
  \label{eq:Xk}
  X_k(\delta) = \sum_{d=1}^{n-1} d^k \alpha_d
\end{equation}
From now on, we will omit the dependence on $\delta$ and just write $X_2$, $X_3$, etc.
\begin{equation}
  \label{eq:Txyx1}
  \eqalign{
  \fl T_{xyx,1}(\delta, d_{n-1}) = T_{xyx,3}(\delta, d_{n-1}) = 
   \frac{T(\delta)\,\alpha_1}{(n-1)(n-2)(n-3)(n-4)}  \cdot \Big(4n^2 - 6n + 2 -X_2(n + 2)\\
   \fl +X_3+(2d_{n-1}^3 - d_{n-1}^2 n - X_2(d_{n-1} - n -3) - 6d_{n-1}^2 +
     7d_{n-1}n - 4n^2 - X_3 - 2d_{n-1}
     +4)\alpha_{d_{n-1}}\Big)\,,
}
\end{equation}
\begin{equation}
  \label{eq:Txyx2}
  \fl T_{xyx,2}(\delta, d_{n-1}) = T(\delta)\,\alpha_1 \,\alpha_{d_{n-1}} (d_{n-1}-1) \, \frac{n^2  + (1- 2d_{n-1})n - 2 + 2d_{n-1}^2 - X_2}{(n-1)(n-2)(n-3)(n-4)}\,,
\end{equation}
where $1\leq d_{n-1}\leq n-1$.
\begin{equation}
  \label{eq:Txyx4}
  \eqalign{
  \fl T_{xyx,4}(\delta, d_0, d_{n-1})+T_{xyx,5}(\delta, d_0, d_{n-1})
  = \frac{T(\delta)\,\alpha_{d_0}}{(n-1)(n-2)(n-3)(n-4)}\cdot \\
  \fl\quad \Big(
  -12d_0^3 - (d_0 - 5)n^2 + X_2(3d_0
    - n - 5) + 30d_0^2 
    + 3(2d_0^2- 7d_0 + 3)n \\
   \fl\quad\phantom{\Big(} + \big[4d_0^3 +
    3d_0^2d_{n-1} + 3d_0d_{n-1}^2 +2d_{n-1}^3 + (d_0 - 5)n^2 -
    X_2(2d_0 + d_{n-1} - n - 5) \\
   \fl\quad\phantom{\Big(} - 11d_0^2 -10d_0d_{n-1} - 9d_{n-1}^2 -
    (3d_0^2 + 2d_0d_{n-1} + d_{n-1}^2 -12d_0 - 9d_{n-1} + 9)n\\
   \fl\quad\phantom{\Big(} - X_3 +
    3d_0 + 5d_{n-1} + 8\big]\alpha_{d_{n-1}}+ X_3 - 8d_0 -
    8\Big)\,,
  }
\end{equation}
\begin{equation}
  \label{eq:Txyx6}
  \eqalign{
  \fl T_{xyx,6}(\delta, d_0, d_{n-1})+T_{xyx,8}(\delta, d_0, d_{n-1}) =
  \frac{T(\delta)\,\alpha_{d_0}}{(n-1)(n-2)(n-3)(n-4)}\cdot\\
  \fl\quad\Big(-6 d_0^3 + X_2 (2 d_0 - n - 4) + 20
    d_0^2 + 2 (d_0^2 - 7 d_0 + 3) n + 4n^2 \\
   \fl\quad\phantom{\Big(} + \big[2 d_0^3 + d_0^2 d_{n-1}
    + d_0 d_{n-1}^2 + 2 d_{n-1}^3 - X_2 (d_0+ d_{n-1} - n - 4) - 7
    d_0^2 - 6 d_0 d_{n-1} \\
   \fl\quad\phantom{\Big(} - 7 d_{n-1}^2 - (d_0^2 +d_{n-1}^2 - 7 d_0 -
    7 d_{n-1} + 6) n - 4 n^2 - X_3 + 3 d_0 + 3 d_{n-1} +6\big]
    \alpha_{d_{n-1}} \\
   \fl\quad\phantom{\Big(} + X_3 - 6 d_0 - 6\Big)\,,
}
\end{equation}
where $2 \leq d_0 \leq n-1$ and $1 \leq d_{n-1} \leq n-1$.
\begin{equation}
  \label{eq:Txyx7}
  \eqalign{
  \fl T_{xyx,7}(\delta, d_0, d_{n-2})+T_{xyx,9}(\delta, d_0, d_{n-2}) = \frac{T(\delta)\,\alpha_{d_0}}{(n-1)(n-2)(n-3)(n-4)}\cdot\\
  \fl\quad\Big(-6 d_0^3 - (d_0 - 1) n^2 + X_2 (d_0
    - 1) + 10 d_0^2 + (4 d_0^2 - 7 d_0 + 3) n\\
   \fl\quad\phantom{\Big(} + \big[2 d_0^2 d_{n-2} + 2
    d_0 d_{n-2}^2 + 2 d_{n-2}^3 + (d_{n-2} - 1) n^2 - 2 d_0^2 - X_2
    (d_{n-2} - 1) - 4 d_0 d_{n-2}\\
   \fl\quad\phantom{\Big(}  - 4 d_{n-2}^2 - (2 d_0 d_{n-2} + 2
    d_{n-2}^2 - 2 d_0 - 5 d_{n-2} + 3) n + 2 d_0 + 2\big] \alpha_{d_{n-2}}
    - 2 d_0 - 2\Big)\,,
}
\end{equation}
where $2 \leq d_0 \leq n-1$ and $2 \leq d_{n-2} \leq n-1$.

Combining the above equations yields for all $d, d_0 \geq 1$:
\begin{equation}
 \label{eq:Txyx_d0_d}
\eqalign{
 \fl T_{xyx}(\delta, d_0, d) = \frac{T(\delta)\,\alpha_{d_0} (\alpha_{d} -
 \delta_{d_0,d})}{(n-1)(n-2)(n-3)(n-4)} \, \Big(6 \, d^{3} + 2 \, {\left(3 \, d - 10\right)}
   d_{0}^{2} + 6 \, d_{0}^{3} \\
  \fl\quad+ (d + d_{0} - 10) n^{2} -
   (3 \, X_{2} - 8) d - 20 \, d^{2} + (6 \, d^{2} -
       3 \, X_{2} - 20 \, d + 8) d_{0} \\
 \fl\quad- {\big[4 \, d^{2} +
       (4 \, d - 21) d_{0} + 4 \, d_{0}^{2} - 2 \, X_{2}
       - 21 \, d + 18\big]} n + 10 \, X_{2} - 2 \, X_{3} +
   16\Big).
}
\end{equation}
Since we need the number of all edge labeled trees in the class $xyx$ to compute the number of error free trees, we compute this formula as well. Combining the above formulas after summing over the free variables results in 
\begin{equation}
  \label{eq:Txyx}
  T_{xyx}(\delta) =
  T(\delta)\ \frac{- 6 n^2 +10 n - 4 + 2 X_2 (n + 1)  - 2 X_3 }{(n-1)(n-2)}.
\end{equation}

\section{The case $T_{xyzx}$}
\label{sec:Txyzx}

For the $xyzx$ case we have the following restrictions
\begin{enumerate}
\item $C=(1)(2)(n-2)\dots$, $0\not\in C$
\item $C=(1)(2)(n-1)\dots$, $0\not\in C$
\item $C=(1)(2)(1)\dots$, $0\not\in C$
\item $C=(1)(2)(2)\dots$, $0\not\in C$

\item $C=(1)(2)(0)\dots$
\item $C=(1)\dots(1)(2)(0)\dots$
\item $C=(1)(2)\dots(2)(0)\dots$
\item $C=(1)\dots(1)(2)\dots(2)(0)\dots$

\item $C=\dots(n-2)(1)(2)(0)\dots$
\item $C=\dots(n-2)(1)\dots(1)(2)(0)\dots$
\item $C=\dots(n-2)(1)(2)\dots(2)(0)\dots$
\item $C=\dots(n-2)(1)\dots(1)(2)\dots(2)(0)\dots$
\item $C=\dots(n-1)(1)(2)(0)\dots$
\item $C=\dots(n-1)(1)\dots(1)(2)(0)\dots$
\item $C=\dots(n-1)(1)(2)\dots(2)(0)\dots$
\item $C=\dots(n-1)(1)\dots(1)(2)\dots(2)(0)\dots$
\end{enumerate}
The first four restrictions are quiete similar to the case $xx_1$,
except that there are two more fixed positions.
Like label $0$, these labels could be any label except the one choosen for $0$. To correct for this have to multiply all counts
by $(n-3)(n-4)$ like we did with $(n-2)$ for label $0$.
\begin{equation}
  \label{eq:Txyzx1}
  \eqalign{
  \fl T_{xyzx,1}(\delta) = T_{xyzx,3}(\delta) = T_{xyzx,4}(\delta) = \\
  \fl\quad T(\delta)\,\alpha_1 \, \frac{-4 \, n^{3} - 6 \, n^{2} + 22 \, n - 12
    + X_{2} {\left(n^{2} + 11 \, n - 4\right)} -2 \, X_{3} {\left(n +
        3\right)} + 2 \, X_{4}}{(n-1)(n-2)(n-3)(n-5)}.
}
\end{equation}
\begin{equation}
  \label{eq:Txyzx2}
   \eqalign{
  \fl T_{xyzx,2}(\delta) =\\
  \fl\quad T(\delta)\,\alpha_1 \, \frac{n^{4} + 10 \, n^{3} - 5 \, n^{2} - 22
    \, n + 16 -6 \, X_{2} {\left(n^{2} + 3 \, n - 2\right)}  + 8 \,
    X_{3} {\left(n + 1\right)} -6 \,
    X_{4}}{(n-1)(n-2)(n-3)(n-4)(n-5)}.
}
\end{equation}
\begin{equation}
  \label{eq:Txyzx5}
   \eqalign{
  \fl
  T_{xyzx,5}(\delta)+T_{xyzx,6}(\delta)+T_{xyzx,7}(\delta)+T_{xyzx,8}(\delta)
  = \frac{T(\delta)}{(n-1)(n-2)(n-3)(n-5)}\cdot\\
  \fl\quad  \Big(-7 \, n^{4} + 22 \, n^{3} + 75 \, n^{2} + 4 \, (2 \, n^{3} + 3 \,
n^{2} - 11 \, n + 6) \alpha_1 - 178 \, n + 88\\
  \fl\quad\phantom{\Big(}+ 2 \, (n^{3} + 5 \, n^{2} - (n^{2} + 11 \, n - 4)
\alpha_1 - 46 \, n + 18) X_{2}\\
\fl\quad\phantom{\Big(}-4 \, (n^{2} - (n + 3) \alpha_1 - 2 \, n
- 11) X_{3}
+2 \, X_{4} (2 \, n - 2 \, \alpha_1 -
    9)\Big).
}
\end{equation}
\begin{equation}
  \label{eq:Txyzx9}
  \eqalign{
  \fl T_{xyzx,9}(\delta) + T_{xyzx,10}(\delta) + T_{xyzx,11}(\delta) + T_{xyzx,12}(\delta)=\frac{T(\delta)}{(n-1)(n-2)(n-3)(n-4)(n-5)}\cdot\\
  \fl\quad \Big(-2 \, (2 \, n^{3} + 3 \, n^{2} - 11 \, n + 6) (n -
3) + 2 \, (2 \, n^{3} + 3 \, n^{2} - 11 \, n + 6)
\alpha_1\\
\fl\quad\phantom{\Big(}+((n^{2} + 11 \, n - 4) (n - 3) -
(n^{2} + 11 \, n - 4) \alpha_1) X_{2}\\
\fl\quad\phantom{\Big(}-2 \, ((n + 3) (n - 3) - (n +
3) \alpha_1) X_{3}
+2 \, X_{4} (n - \alpha_1 -
3)\Big).
}
\end{equation}
\begin{equation}
  \label{eq:Txyzx13}
\eqalign{
  \fl T_{xyzx,13}(\delta) + T_{xyzx,14}(\delta) + T_{xyzx,15}(\delta) + T_{xyzx,16}(\delta)=\frac{T(\delta)}{(n-1)(n-2)(n-3)(n-5)}\cdot\\
  \fl\quad\Big(n^{4} + 10 \, n^{3} - 5 \, n^{2} - 22 \, n + 16) (n
- 4) - (n^{4} + 10 \, n^{3} - 5 \, n^{2} - 22 \, n +
16) \alpha_1\\
\fl\quad\phantom{\Big(}-6 \, ((n^{2} + 3 \, n - 2) (n - 4) -
(n^{2} + 3 \, n - 2) \alpha_1) X_{2}\\
\fl\quad\phantom{\Big(}+8 \, ((n + 1) (n - 4) - (n +
1) \alpha_1) X_{3}
-6 \, X_{4} (n - \alpha_1 - 4)\Big).
}
\end{equation}
Combining the above formulas results in 
\begin{equation}
  \label{eq:Txyzx}
  \fl T_{xyzx}(\delta) =
  T(\delta)\ \frac{-10 \, n^{3} - 12 \, n^{2} + 50 \, n - 28 + 3 \, (n^{2} + 9 \, n - 4) X_{2} -2 \, X_{3} (3 \, n + 7) + 6 \, X_{4}}{(n-1)(n-2)(n-3)}.
\end{equation}

\section*{References}

\bibliographystyle{unsrt} 
\bibliography{animals,math,mertens}

\end{document}